\documentclass[11pt, a4paper]{amsart}
\usepackage{graphicx,multirow,array,amsmath,amssymb}

\newtheorem{theorem}{Theorem}

\newtheorem{lemma}[theorem]{Lemma}

\begin{document}

\title{Maximal independent sets on a grid graph}

\author[S. Oh]{Seungsang Oh}
\address{Department of Mathematics, Korea University, Seoul 02841, Korea}
\email{seungsang@korea.ac.kr}

\thanks{This research was supported by the National Research Foundation of Korea(NRF) grant funded
by the Korea government(MSIP) (No. NRF-2014R1A2A1A11050999).}

\begin{abstract}
An independent vertex set of a graph is a set of vertices of the graph in which no two vertices are adjacent,
and a maximal independent set is one that is not a proper subset of any other independent set.
In this paper we count the number of maximal independent sets of vertices on a complete rectangular grid graph.
More precisely, we provide a recursive matrix-relation producing the partition function 
with respect to the number of vertices.
The asymptotic behavior of the maximal hard square entropy constant is also provided.
We adapt the state matrix recursion algorithm,
recently invented by the author to answer various two-dimensional regular lattice 
model problems in enumerative combinatorics and statistical mechanics.
\end{abstract}

\maketitle

\section{Introduction} \label{sec:intro}

In graph theory, many problems involve subsets of the vertices of a graph
that satisfy certain restrictions based on the adjacency relations within the graphs~\cite{MS1, MS3}.
Among them, counting all maximal independent sets of a given graph is 
one that has attracted considerable attention.
In a graph $G$, 
an {\em independent vertex set\/} is a subset $S$ of its vertex set $V(G)$ such that 
there is no edge of $G$ between any two vertices of $S$.
A {\em maximal independent set\/} (MIS) is an independent vertex set
that is not a proper subset of any other independent vertex set.
In other words, it is a set $S$ such that every edge of the graph has at least one endpoint not in $S$
and every vertex not in $S$ has at least one neighbor in $S$.

Erd\"{o}s and Moser raised the problem of determining the maximum value of the number of MISs 
in a general graph with $n$ vertices and those graphs having this maximum value.
Moon and Moser~\cite{MM} presented that a graph can have at most $3^{n/3}$ MISs and that 
there are graphs achieving this many.
Later Griggs, Grinstead and Guichard~\cite{GGG} improved this result for connected graphs. 
This problem has been extensively studied for various classes of graphs, 
including trees~\cite{Sa, Wi} and graphs with at most $r$ cycles~\cite{GKSV, SV}.

Recently several significant enumeration problems regarding various combinatorial objects 
on the $m \! \times \! n$ grid graph were solved by means of the {\em state matrix recursion algorithm\/},
originated from~\cite{OHLL} and later developed by the author~\cite{OhD1}.
As one of the most interesting applications, 
this algorithm provides a recursive matrix-relation producing the exact number of
independent vertex sets on the $m \! \times \! n$ grid graph in the preceding paper~\cite{OhV1}.
This is well known as the Hard Square Problem or the Merrifield-Simmons index~\cite{MS1, MS3}.
This index is an important topological index for the study of the relation 
between molecular structure and physical/chemical properties of certain hydrocarbon compounds, 
such as the correlation with boiling points~\cite{GP}.
A good summary of results on the Merrifield-Simmons index of graphs 
can be found in the survey paper~\cite{WG}.

In this paper, we apply the state matrix recursion algorithm 
to calculate the number of MISs on the $m \! \times \! n$ rectangular grid graph $\mathbb{Z}_{m \times n}$
that is the most interesting two-dimensional regular lattice.
A MIS on $\mathbb{Z}_{m \times n}$ is drawn in Figure~\ref{fig:MIS}~(a).
An independent vertex set is often represented by a hard square lattice gas with nearest-neighbor exclusion.
In a MIS, all vertices must be covered by hard squares as in Figure~\ref{fig:MIS}~(b).
Up to now there are only partial results~\cite{Eu, EOS} on counting MISs in $\mathbb{Z}_{m \times n}$.

\begin{figure}[h]
\includegraphics{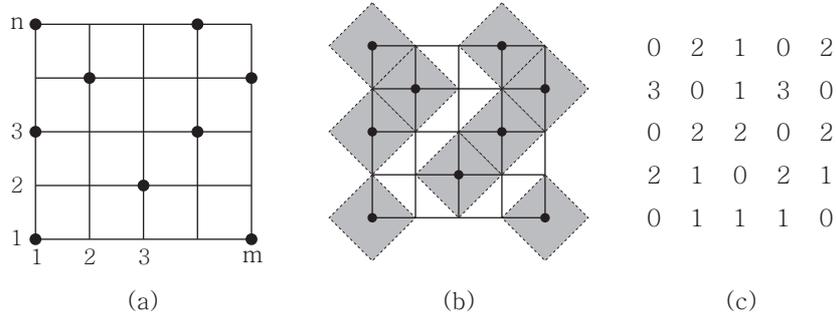}
\caption{(a) A MIS on $\mathbb{Z}_{m \times n}$.
(b) A maximal hard square lattice gas.
(c) Another 0, 1, 2, 3, 4 array form.}
\label{fig:MIS}
\end{figure}

The partition function of MISs at activity $z$ on $\mathbb{Z}_{m \times n}$ is defined by
$$P_{m \times n}(z) = \sum \, k(t) \, z^t,$$
where $k(t)$ is the number of MISs consisting of $t$ vertices.
Then, the number of MISs is
$$\sigma(\mathbb{Z}_{m \times n}) = P_{m \times n}(1).$$

It is of interest to note that $\sigma(\mathbb{Z}_{m \times n})$ is indeed the number of $m \times n$ arrays 
of five digits 0, 1, 2, 3, 4 where each entry equals the number of its horizontal and vertical zero neighbors, 
where 0's are located at the place of a MIS as in Figure~\ref{fig:MIS}~(c).
The numbers of $\sigma(\mathbb{Z}_{m \times n})$ are registered in
Sloane's On-Line Encyclopedia of Integer Sequences~\cite{Sl}, namely A197054.

By virtue of the state matrix recursion algorithm,
we present a recursive formula for this partition function.
Hereafter $\mathbb{O}_k$ denotes the $3^k \! \times \! 3^k$ zero-matrix,
and $\stackrel{m}{\otimes} A$ denotes the $m$-fold tensor product\footnote{
The matrix describing the tensor product $A \otimes B$ is the Kronecker product of the two matrices.
For example,
{\scriptsize 
$$ \begin{bmatrix} 1 & 2 \\ 3 & 4 \end{bmatrix} \otimes \begin{bmatrix} 5 \\ 6 \end{bmatrix} =
\begin{bmatrix} 1 \cdot \begin{bmatrix} 5 \\ 6 \end{bmatrix} & 2 \cdot \begin{bmatrix} 5 \\ 6 \end{bmatrix} \\ 
3 \cdot \begin{bmatrix} 5 \\ 6 \end{bmatrix} & 4 \cdot \begin{bmatrix} 5 \\ 6 \end{bmatrix} \end{bmatrix} =
\begin{bmatrix} 1 \cdot 5 & 2 \cdot 5 \\ 1 \cdot 6 & 2 \cdot 6 \\ 
3 \cdot 5 & 4 \cdot 5 \\ 3 \cdot 6 & 4 \cdot 6 \end{bmatrix} =
\begin{bmatrix} 5 & 10 \\ 6 & 12 \\ 15 & 20 \\ 18 & 24 \end{bmatrix}. $$}
} of a matrix $A$.

\begin{theorem} \label{thm:main}
The partition function $P_{m \times n}(z)$ is the unique entry of the $1 \! \times \! 1$ matrix
$$D_m \cdot (A_m + B_m)^n \cdot E_m,$$
where $A_m$ and $B_m$ are $3^m \! \times \! 3^m$ matrices recursively defined by
$$A_{k+1} = \begin{bmatrix} \mathbb{O}_k & \mathbb{O}_k & z \, C_k \\
\mathbb{O}_k & \mathbb{O}_k & \mathbb{O}_k \\
\mathbb{O}_k & \mathbb{O}_k & \mathbb{O}_k \end{bmatrix}, \
B_{k+1} = \begin{bmatrix} \mathbb{O}_k & \mathbb{O}_k & \mathbb{O}_k \\
A_k \! + \! B_k & A_k & \mathbb{O}_k \\
A_k \! + \! B_k & A_k \! + \! B_k & \mathbb{O}_k \end{bmatrix}$$
$$\text{and } \ C_{k+1} = \begin{bmatrix} \mathbb{O}_k & \mathbb{O}_k & \mathbb{O}_k \\
A_k \! + \! B_k &A_k \! + \! B_k & \mathbb{O}_k \\
A_k \! + \! B_k & A_k \! + \! B_k & \mathbb{O}_k \end{bmatrix},$$
for $k=0, \dots, m \! - \! 1$,
starting with 
$A_0 = \begin{bmatrix} 0 \end{bmatrix}$
and $B_0 = C_0 = \begin{bmatrix} 1 \end{bmatrix}$,
and $D_m$ and $E_m$ are respectively $1 \! \times \! 3^m$ and $3^m \! \times \! 1$ matrices defined by
$$D_m = \ \stackrel{m}{\otimes} \begin{bmatrix} 1 & 1 & 0 \end{bmatrix} \text{ and } \
E_m = \ \stackrel{m}{\otimes} \begin{bmatrix} 0 \\ 1 \\ 1 \end{bmatrix}.$$
\end{theorem}

This partition function gives the following significant consequences:
The lowest degree of $P_{m \times n}(z)$ indicates the minimum number of vertices to produce a MIS,
and its coefficient is the number of MISs with fewest set. 

Note that the square matrices $A_m$ and $B_m$ are exponentially large;
this is exactly what one would expect just from applying the standard transfer matrix method
used for other lattice models.
But these matrices are significantly smaller than what would be obtained 
from naive application of the transfer matrix method (See Table~\ref{tab:list}).

\begin{table}[h]
{\footnotesize \begin{tabular}{ccccccccc}      \hline \hline
 & $n=1$ & $n=2$ & $n=3$ & $n=4$ & $n=5$ & $n=6$ & $n=7$ & $n=8$ \\    \hline
$m=1$ & 1 &   &   &   &   &   &   & \\
$m=2$ & 2 & 2 &   &   &   &   &   & \\
$m=3$ & 2 & 4 & 10 &   &   &   &   & \\
$m=4$ & 3 & 6 & 18 & 42 &   &   &   & \\
$m=5$ & 4 & 10 & 38 & 108 & 358 &   &   & \\
$m=6$ & 5 & 16 & 78 & 274 & 1132 & 4468 &   &  \\
$m=7$ & 7 & 26 & 156 & 692 & 3580 & 17742 & 88056 &  \\
$m=8$ & 9 & 42 & 320 & 1754 & 11382 & 70616 & 439338 & 2745186 \\
$m=9$ & 12 & 68 & 654 & 4442 & 36270 & 281202 & 2192602 & 17155374 \\
$m=10$ & 16 & 110 & 1326 & 11248 & 114992 & 1117442 & 10912392 & 106972582 \\  \hline \hline
\end{tabular}}
\vspace{4mm}
\caption{List of the exact numbers of $\sigma(\mathbb{Z}_{m \times n})$}
\label{tab:list}
\end{table}

We are turning now to the growth rate per vertex of the number of MISs 
$\sigma(\mathbb{Z}_{m \times n})$ as defined by
$$ \lim_{m, \, n \rightarrow \infty} (\sigma(\mathbb{Z}_{m \times n}))^{\frac{1}{mn}} = \kappa.$$
We call this limit $\kappa$ the {\em maximal hard square entropy constant\/}.
A two-dimensional application of Fekete's lemma gives 
a mathematical proof of the existence of the limit.

\begin{theorem}\label{thm:growth}
The maximal hard square entropy constant $\kappa$ exists.
More precisely, 
$$ \kappa = \sup_{m, \, n \geq 1} (\sigma(\mathbb{Z}_{m \times n}))^{\frac{1}{(m+1)(n+1)}}.$$
\end{theorem}

We approximate the maximal hard square entropy constant as little greater than 
$1.225084\cdots$ from $\sigma(\mathbb{Z}_{8 \times 380}) = 2.0932\cdots \! \times \! 10^{302}$.

\section{State matrix recursion algorithm}  

In this section we prove Theorem~\ref{thm:main} by means of the state matrix recursion algorithm.
This algorithm is divided into three stages.
The first stage is devoted to the installation of the mosaic system for MISs on the grid graph.
Note that the original construction of the mosaic system for quantum knots 
was invented by Lomonaco and Kauffman to represent an actual physical quantum system~\cite{LK}.
Recently, the author {\em et al\/}. have developed a state matrix argument for knot mosaic enumeration
in a series of papers \cite{HLLO, HO, Oh1, OHLL, OHLLY}.
The state matrix recursion algorithm is a well-formalized version of this argument
to answer various two-dimensional square lattice model problems 
in enumerative combinatorics and statistical mechanics~\cite{OhD1, OhV1}.
In the second stage, we find recursive matrix-relations producing the exact enumeration.
Our proofs of two lemmas in this stage parallel those of Lemmas~3 and 4 in~\cite{OhV1}, 
with slight modification to MISs.
In the third stage, we analyze the state matrix obtained in the second stage to complete the proof.

\subsection{Stage 1: Conversion to the MIS mosaic system} \hspace{10mm}

In this paper, we consider the sixteen {\em mosaic tiles\/} illustrated in Figure~\ref{fig:tile}.
Their side edges are labeled with three letters a, b and c.
Note that the dot in the first mosaic tile $T_1$ indicates a vertex in a MIS.
In detail, $T_1$ has four side edges labeled with one letter a only. 
All the other fifteen mosaic tiles have side edges labeled with letters b and c, 
except that four b's are not allowed.

\begin{figure}[h]
\includegraphics{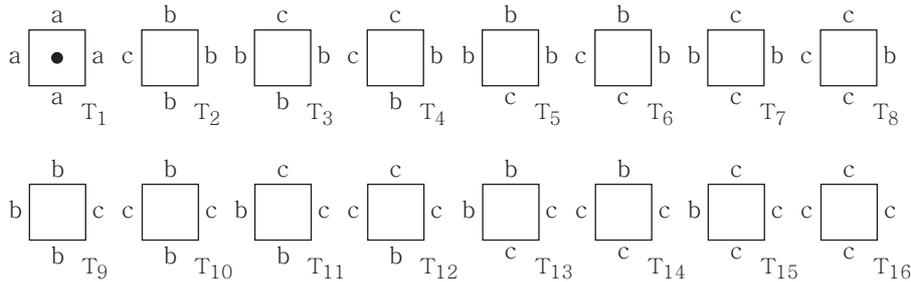}
\caption{Sixteen mosaic tiles labeled with three letters}
\label{fig:tile}
\end{figure}

For positive integers $m$ and $n$,
an {\em $m \! \times \! n$--mosaic\/} is an $m \! \times \! n$ rectangular array $M = (M_{ij})$ of those tiles,
where $M_{ij}$ denotes the mosaic tile placed at the $i$th column from `left' to `right'
and the $j$th row from `bottom' to `top'. 
We are exclusively interested in mosaics whose tiles match each other properly to represent MISs.
For this purpose we consider the following rules. \\

\noindent {\bf Adjacency rule (abc--cba type):\/}
Adjacent edges of adjacent mosaic tiles in a mosaic
must be labeled with any of the following pairs of letters: a/c, b/b. \\

\noindent {\bf Boundary state requirement:\/}
All boundary edges in a mosaic are labeled with letters a and b (but, not c). \\

The meaning of the abc--cba type is that 
if the labeled state of one side of a pair of adjacent edges is a (b or c), then the other side is c (b or a, respectively).
As illustrated in Figure~\ref{fig:conversion},
every MIS in $\mathbb{Z}_{m \times n}$ can be converted into 
an $m \! \times \! n$--mosaic which satisfies the two rules.
According to the adjacency rule, we avoid putting two mosaic tiles $T_1$ next to each other
(to be an independent vertex set),
and guarantee that each of fifteen mosaic tiles $T_2 \! \sim \! T_{16}$ has
at least one $T_1$ as neighbors (to be maximal). 

\begin{figure}[h]
\includegraphics{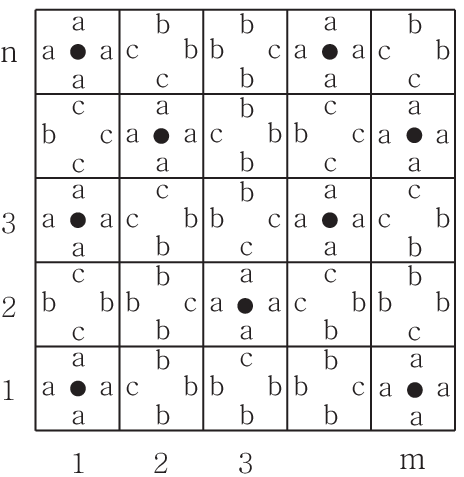}
\caption{Conversion of the MIS in Figure~\ref{fig:MIS} to a MIS $m \! \times \! n$--mosaic}
\label{fig:conversion}
\end{figure}

A mosaic is said to be {\em suitably adjacent\/} if any pair of mosaic tiles
sharing an edge satisfies the adjacency rule.
A suitably adjacent $m \! \times \! n$--mosaic is called a {\em MIS $m \! \times \! n$--mosaic\/} 
if it additionally satisfies the boundary state requirement.
The boundary state requirement guarantees the uniqueness of a MIS $m \! \times \! n$--mosaic
representing a given MIS.
The following one-to-one conversion arises naturally. \\

\noindent {\bf One-to-one conversion:\/}
There is a one-to-one correspondence between  
MISs on $\mathbb{Z}_{m \times n}$ and MIS $m \! \times \! n$--mosaics.
Furthermore, the number of vertices in a MIS is equal to
the number of $T_1$ mosaic tiles in the corresponding MIS $m \! \times \! n$--mosaic.

\subsection{Stage 2: State matrix recursion formula} \hspace{10mm}

Let $p \leq m$ and $q \leq n$ be positive integers.
Consider a suitably adjacent $p \! \times \! q$--mosaic $M$, possibly labeled c on boundary edges.
We use $v(M)$ to denote the number of appearances of $T_1$ tiles in $M$.
A {\em state\/} is a finite sequence of three letters a, b and c.
The {\em $b$--state\/} $s_b(M)$ ({\em $t$--state\/} $s_t(M)$) is 
the state of length $p$ obtained by reading off letters on the bottom (top, respectively) 
boundary edges from right to left, 
and the {\em $l$--state\/} $s_l(M)$ ({\em $r$--state\/} $s_r(M)$) is 
the state of length $q$ on the left (right, respectively) boundary edges from top to bottom.
For example, the MIS $5 \! \times \! 5$--mosaic drawn in Figure~\ref{fig:conversion} has four state indications:
$s_b =$ abbba, $s_t =$ babba, $s_l =$ ababa, and $s_r =$ babba.

Given a triple $\langle s_r, s_b, s_t \rangle$ of $r$--, $b$-- and $t$--states,
we associate the {\em state polynomial\/}:
$$ Q_{\langle s_r, s_b, s_t  \rangle}(z) = \sum k(v) \, z^v, $$
where $k(v)$ is the number of all suitably adjacent $p \! \times \! q$--mosaics $M$ such that 
$v(M)=v$, $s_r(M) = s_r$, $s_b(M) = s_b$, $s_t(M) = s_t$
and $s_l(M)$ is any state of length $q$ consisting of only two letters a and b.
The last condition for $s_l(M)$ is due to the left boundary state requirement.

Now we focus on mosaics of width 1.
Consider a suitably adjacent $p \! \times \! 1$--mosaic ($p \leq m$), which is called a {\em bar mosaic\/}.
Bar mosaics of length $p$ have $3^p$ kinds of $b$-- and $t$--states, especially called {\em bar states\/}.
We arrange all bar states in two ways as follows: for example if $p=2$,
the abc-ordered state set as aa, ab, ac, ba, bb, bc, ca, cb and cc (the lexicographic order),
and the cba-ordered state set as cc, cb, ca, bc, bb, ba, ac, ab and aa (the reverse lexicographic order).
For $1 \leq i \leq 3^p$, let $\epsilon^p_i$ and $\lambda^p_i$ denote the $i$th bar states of length $p$
among the abc- and cba-ordered state sets, respectively.

The {\em bar state matrix\/} $X_p$ ($X = A, B, C$)
for the set of suitably adjacent bar mosaics of length $p$ is a $3^p \! \times \! 3^p$ matrix $(m_{ij})$ given by  
$$ m_{ij} = Q_{\langle \text{x}, \epsilon^p_i, \lambda^p_j \rangle}(z) $$
where x $=$ a, b, c, respectively.
We remark that information on suitably adjacent bar mosaics is completely encoded 
in the three bar state matrices $A_p$, $B_p$ and $C_p$.

\begin{lemma} \label{lem:bar}
The bar state matrices $A_p$, $B_p$ and $C_p$ are obtained by the recurrence relations:
$$A_{k+1} = \begin{bmatrix} \mathbb{O}_k & \mathbb{O}_k & z \, C_k \\
\mathbb{O}_k & \mathbb{O}_k & \mathbb{O}_k \\
\mathbb{O}_k & \mathbb{O}_k & \mathbb{O}_k \end{bmatrix}, \mbox{ }
B_{k+1} = \begin{bmatrix} \mathbb{O}_k & \mathbb{O}_k & \mathbb{O}_k \\
A_k \! + \! B_k & A_k & \mathbb{O}_k \\
A_k \! + \! B_k & A_k \! + \! B_k & \mathbb{O}_k \end{bmatrix}$$
$$\text{and } \ C_{k+1} = \begin{bmatrix} \mathbb{O}_k & \mathbb{O}_k & \mathbb{O}_k \\
A_k \! + \! B_k &A_k \! + \! B_k & \mathbb{O}_k \\
A_k \! + \! B_k & A_k \! + \! B_k & \mathbb{O}_k \end{bmatrix}$$
with seed matrices
$A_1 = {\small \begin{bmatrix} 0 & 0 & z \\ 0 & 0 & 0 \\ 0 & 0 & 0 \end{bmatrix}}, \
B_1 = {\small \begin{bmatrix} 0 & 0 & 0 \\ 1 & 0 & 0 \\ 1 & 1 & 0 \end{bmatrix}} \mbox{ and\, }
C_1 = {\small \begin{bmatrix} 0 & 0 & 0 \\ 1 & 1 & 0 \\ 1 & 1 & 0 \end{bmatrix}}.$
\end{lemma}

Note that we may start with matrices
$A_0 = \begin{bmatrix} 0 \end{bmatrix}$ and $B_0 = C_0 = \begin{bmatrix} 1 \end{bmatrix}$
instead of $A_1$, $B_1$ and $C_1$.

\begin{proof}
The following proof parallels the inductive proof of \cite[Lemma~3]{OhV1} with slight modification.
By observing the sixteen mosaic tiles,
we find the first bar state matrices $A_1$, $B_1$ and $C_1$ as in the lemma.
For example, $(1,3)$-entry of $A_1$ is 
$$ Q_{\langle \text{a}, \epsilon^1_1, \lambda^1_3 \rangle}(z) =
Q_{\langle \text{a}, \text{a}, \text{a} \rangle}(z) = z $$
since exactly one mosaic tile $T_1$ satisfies this requirement.

Assume that the bar state matrices $A_k$, $B_k$ and $C_k$ satisfy the statement.
Consider the matrix $B_{k+1}$, which is of size $3^{k+1} \! \times \! 3^{k+1}$.
Partition this matrix into nine block submatrices of size $3^k \! \times \! 3^k$, 
and consider the 21-submatrix of $B_{k+1}$ 
i.e., the $(2,1)$-component in the $3 \! \times \! 3$ array of the nine blocks.
Due to the abc- and cba-orders, the $(i,j)$-entry of the 21-submatrix is the state polynomial 
$Q_{\langle \text{b}, \text{b}\epsilon^k_i, \text{c}\lambda^k_j \rangle}(z)$
where b$\epsilon^k_i$ (similarly c$\lambda^k_j$) is a bar state of length $k \! + \! 1$
obtained by concatenating two bar states b and $\epsilon^k_i$.
A suitably adjacent $(k \! + \! 1) \! \times \! 1$--mosaic corresponding to 
this triple $\langle \text{b}, \text{b}\epsilon^k_i, \text{c}\lambda^k_j \rangle$
must have a tile either $T_3$ or $T_4$ at the place of the rightmost mosaic tile, 
and so its second rightmost tile must have $r$--state b or a, respectively, by the adjacency rule.
By considering the contribution of the rightmost tile $T_3$ or $T_4$ to the state polynomial,
one easily gets
$$Q_{\langle \text{b}, \text{b}\epsilon^k_i, \text{c}\lambda^k_j \rangle}(z) = 
(i,j)\text{-entry of } (A_k + B_k).$$
Thus the 21-submatrix of $B_{k+1}$ is $A_k + B_k$.
Using the same argument, 
we derive Table~\ref{tab:barset} presenting all possible twenty seven cases as desired.
\end{proof}

\begin{table}[h]
\begin{tabular}{ccccc}      \hline \hline
 & \ \ & {\em Submatrix for\/} $\langle s_r, s_b, s_t \rangle$ & {\em Rightmost tile\/} &
{\em Submatrix\/} \\    \hline
\multirow{1}{3mm}{$A_{k+1}$}
 & & 13-submatrix $\langle \text{a}, \text{a} \! \cdot \! \cdot,\text{a} \! \cdot \! \cdot \rangle$
 & $T_1$ & $z \, C_k$ \\  \hline
\multirow{4}{3mm}{$B_{k+1}$}
 & & 21-submatrix $\langle \text{b}, \text{b} \! \cdot \! \cdot,\text{c} \! \cdot \! \cdot \rangle$ 
 & $T_3$, $T_4$ &  $A_k + B_k$ \\
 & & 22-submatrix $\langle \text{b}, \text{b} \! \cdot \! \cdot,\text{b} \! \cdot \! \cdot \rangle$ 
 & $T_2$ & $A_k$ \\
 & & 31-submatrix $\langle \text{b}, \text{c} \! \cdot \! \cdot,\text{c} \! \cdot \! \cdot \rangle$ 
 & $T_7$, $T_8$ & $A_k + B_k$ \\
 & & 32-submatrix $\langle \text{b}, \text{c} \! \cdot \! \cdot,\text{b} \! \cdot \! \cdot \rangle$ 
 & $T_5$, $T_6$ & $A_k + B_k$ \\  \hline
\multirow{4}{3mm}{$C_{k+1}$}
 & & 21-submatrix $\langle \text{c}, \text{b} \! \cdot \! \cdot,\text{c} \! \cdot \! \cdot \rangle$ 
 & $T_{11}$, $T_{12}$ & $A_k + B_k$ \\
 & & 22-submatrix $\langle \text{c}, \text{b} \! \cdot \! \cdot,\text{b} \! \cdot \! \cdot \rangle$ 
 & $T_9$, $T_{10}$ & $A_k + B_k$ \\
 & & 31-submatrix $\langle \text{c}, \text{c} \! \cdot \! \cdot,\text{c} \! \cdot \! \cdot \rangle$ 
 & $T_{15}$, $T_{16}$ & $A_k + B_k$ \\
 & & 32-submatrix $\langle \text{c}, \text{c} \! \cdot \! \cdot,\text{b} \! \cdot \! \cdot \rangle$ 
 & $T_{13}$, $T_{14}$ & $A_k + B_k$ \\  \hline
 & & The other 18 cases & None & $\mathbb{O}_k$ \\     \hline \hline
\end{tabular}
\vspace{4mm}
\caption{27 submatrices of $A_{k+1}$, $B_{k+1}$ and $C_{k+1}$}
\label{tab:barset}
\end{table}

Now we extend to mosaics of any width.
The {\em state matrix\/} $Y_{m \times q}$ for the set of suitably adjacent $m \! \times \! q$--mosaics ($q \leq n$) 
is a $3^m \! \times \! 3^m$ matrix $(y_{ij})$ given by 
$$ y_{ij} = \sum Q_{\langle s_r, \epsilon^m_i, \lambda^m_j \rangle}(z), $$
where the summation is taken over all $r$--states $s_r$ of length $q$ consisting of only two letters a and b.
The summation condition of $s_r$ is due to the right boundary state requirement.

\begin{lemma} \label{lem:mn}
The state matrix $Y_{m \times n}$ is obtained by
$$Y_{m \times n} = (A_m + B_m)^n.$$
\end{lemma}

\begin{proof}
The following proof parallels the inductive proof of \cite[Lemma~4]{OhV1} with slight modification.
For $n=1$, $Y_{m \times 1} =A_m + B_m$
since $Y_{m \times 1}$ counts suitably adjacent $m \! \times \! 1$--mosaics 
with $l$-- and $r$--states consisting of letters a or b.
Assume that $Y_{m \times k} = (A_m + B_m)^k$.
Consider a suitably adjacent $m \! \times \! (k \! + \! 1)$--mosaic $M^{m \times (k+1)}$
with $r$-- and $l$--states consisting of letters a and b.
Split it into two suitably adjacent $m \! \times \! k$-- and $m \! \times \! 1$--mosaics
$M^{m \times k}$ and $M^{m \times 1}$ by tearing off the topmost bar mosaic.
There is a certain relation between the $t$--state of $M^{m \times k}$ 
and the $b$--state of $M^{m \times 1}$ as shown in Figure~\ref{fig:expand}.
To satisfy the adjacency rule, the letters a and c are changed by c and a, respectively, 
from one state to the other.
The key point is that, for some $r = 1, \dots, 3^m$, 
$s_t(M^{m \times k}) = \lambda^m_r$ and $s_b(M^{m \times 1}) = \epsilon^m_r$.

\begin{figure}[h]
\includegraphics{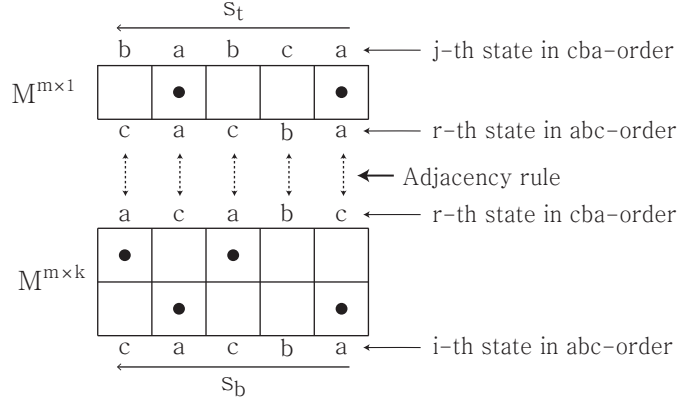}
\caption{Expanding $M^{m \times k}$ to $M^{m \times (k+1)}$}
\label{fig:expand}
\end{figure}

Let $Y_{m \times (k+1)} = (y_{ij})$, $Y_{m \times k} = (y_{ij}')$ and $Y_{m \times 1} = (y_{ij}'')$.
Note that $y_{ij}$ is the state polynomial 
for the set of suitably adjacent $m \! \times \! (k \! + \! 1)$--mosaics $M$ 
which admit splittings into $M^{m \times k}$ and $M^{m \times 1}$ satisfying
$s_b(M) = s_b(M^{m \times k}) = \epsilon^m_i$,
$s_t(M) = s_t(M^{m \times 1}) = \lambda^m_j$, 
and for some $r = 1, \dots, 3^m$,
$s_t(M^{m \times k}) = \lambda^m_r$ and $s_b(M^{m \times 1}) = \epsilon^m_r$.
Obviously, all their $l$-- and $r$--states consist of letters a and b.
Since all $3^m$ kinds of bar states arise as states of these $m$ horizontal adjacent edges,
$$ y_{ij} = \sum^{3^m}_{r=1} y_{ir}'  \cdot  y_{rj}''. $$
This implies
$$ Y_{m \times (k+1)} = Y_{m \times k} \cdot Y_{m \times 1} = (A_m + B_m)^{k+1}, $$
and the induction step is finished.
\end{proof}

\subsection{Stage 3: State matrix analyzing} \hspace{10mm}

\begin{proof}[Proof of Theorem~\ref{thm:main}.]
The $(i,j)$-entry of $Y_{m \times n}$ is the state polynomial for the set of 
suitably adjacent $m \! \times \! n$--mosaics $M$ with
$s_b(M) = \epsilon^m_i$, $s_t(M) = \lambda^m_j$ and $r$-- and $l$--states consisting of letters a and b.
According to the boundary state requirement,
MISs in $\mathbb{Z}_{m \times n}$ are converted into 
suitably adjacent $m \! \times \! n$--mosaics $M$ with $b$--, $t$--, $r$-- and $l$--states
consisting of letters a and b (but not c).
Thus the partition function $P_{m \times n}(z)$ is the sum of all entries of $Y_{m \times n}$ associated to 
$b$-- and $t$--states only consisting of letters a and b.

Now we define the two matrices $D_m$ and $E_m$ as in Theorem~\ref{thm:main}.
Obviously, $D_m \cdot Y_{m \times n}$ is the $1 \! \times \! 3^m$ matrix obtained from $Y_{m \times n}$ by
nullifying each $i$th row where the corresponding $i$th state in the abc-order has at least one letter of c,
followed by summing each column.
Again, $D_m \cdot Y_{m \times n} \cdot E_m$ is the $1 \! \times \! 1$ matrix 
obtained from $D_m \cdot Y_{m \times n}$ by
nullifying each $j$th column where the corresponding $j$th state in the cba-order has at least one letter of c,
followed by summing the unique row.
Therefore we get the partition function $P_{m \times n}(z)$ from the unique entry of
$D_m \cdot Y_{m \times n} \cdot E_m$.
This fact combined with Lemmas~\ref{lem:bar} and~\ref{lem:mn} completes the proof.
\end{proof}

\section{Maximal hard square entropy constant} \label{sec:growth}

We will need the following result called Fekete's lemma with slight modification.

\begin{lemma}{\cite[Lemma~7]{OhD1}}  \label{lem:growth}
If a double sequence $\{ a_{m,n} \}_{m, \, n \in \, \mathbb{N}}$ with $a_{m,n} \geq 1$ satisfies
$a_{m_1,n} \cdot a_{m_2,n} \leq a_{m_1 + m_2+1,n}$ and $a_{m,n_1} \cdot a_{m,n_2} \leq a_{m,n_1 + n_2+1}$
for all $m$, $m_1$, $m_2$, $n$, $n_1$ and $n_2$,
then
$$ \lim_{m, \, n \rightarrow \infty} (a_{m,n})^{\frac{1}{mn}} = 
\sup_{m, \, n \geq 1} (a_{m,n})^{\frac{1}{(m+1)(n+1)}},$$
provided that the supremum exists.
\end{lemma} 

\begin{proof}[Proof of Theorem~\ref{thm:growth}.]
Obviously, $\sigma(\mathbb{Z}_{m \times n})$ is at least 1 for all $m$, $n$.
For any two MIS $m_1 \! \times \! n$-- and $m_2 \! \times \! n$--mosaics,
we can always create a new MIS $(m_1 \! + \! m_2 \! + \! 1) \! \times \! n$--mosaic
by inserting a proper $1 \! \times \! n$--mosaic between them as in Figure~\ref{fig:growth}.
More precisely, we put $T_1$ tiles in all other places in each consecutive shaded region.
Note that all tiles in the $1 \! \times \! n$--mosaic whose neither left nor right neighbors
are $T_1$ tiles are shaded. 
Therefore $\sigma(\mathbb{Z}_{m_1 \times n}) \cdot \sigma(\mathbb{Z}_{m_2 \times n}) \leq 
\sigma(\mathbb{Z}_{(m_1+m_2+1) \times n})$, and similarly for the other index.
Since we use total sixteen mosaic tiles at each site,
$\sup_{m, \, n} (\sigma(\mathbb{Z}_{m \times n}))^{\frac{1}{(m+1)(n+1)}} \leq 16$, and now apply Lemma~\ref{lem:growth}.
\end{proof}

\begin{figure}[h]
\includegraphics{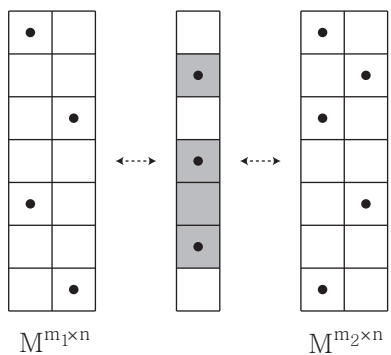}
\caption{Adjoining two MIS mosaics}
\label{fig:growth}
\end{figure}

\end{document}